\documentclass[11pt]{article}

\usepackage[margin=3.4cm, marginparwidth=3.3cm
]{geometry}

\usepackage{amsfonts,amsmath,amssymb,amsthm,xcolor,mathrsfs}
\usepackage{microtype}
\usepackage{comment}
\usepackage{hyperref}
\usepackage{graphicx}
\usepackage{dsfont}
\newtheorem{theorem}{Theorem}[section]
\newtheorem{definition}[theorem]{Definition}
\newtheorem{proposition}[theorem]{Proposition}

\newtheorem{lemma}[theorem]{Lemma}
\newtheorem{remark}[theorem]{Remark}
\newtheorem{notation}[theorem]{Notation}

\numberwithin{equation}{section}

\newcommand{\keywords}[1]{\par\noindent\textit{Key words and phrases. } #1\par}
\newcommand{\msc}[1]{\par\noindent\textit{2020 Mathematics Subject Classification. } #1\par}

\newcommand{\B}{\mathcal{B}}

\newcommand{\R}{\mathbb{R}}
\newcommand{\T}{\mathbb{T}}
\newcommand{\N}{\mathbb{N}}
\newcommand{\Z}{\mathbb{Z}}

\newcommand{\ex}{\mathbb{E}}

\usepackage{authblk}

\usepackage[textsize=tiny]{todonotes}
\author{{Ioannis Gasteratos\thanks{TU Berlin, \textit{email address}: i.gasteratos@tu-berlin.de}}\;\;and {Zachary Selk\thanks{FSU, \textit{email address}: zselk@fsu.edu}}}

\title{On the Onsager-Machlup functional of the $\Phi^4$-measure}

\date{\today}
\begin{document}

\maketitle

\begin{abstract}
We investigate the existence of generalised densities for the $\Phi^4_d$ $(d=1,2,3)$ measures, in finite volume, through the lens of Onsager-Machlup (OM) functionals. The latter are rigorously defined for measures on metric spaces as limiting ratios of small ball probabilities. In one dimension, we show that the standard OM functional of the $\Phi^4_1$ measure coincides with the $\Phi^4$ action as expected. In two dimensions, we show that OM functionals of the $P(\Phi)_2$ measures agree with the corresponding actions, by considering ``enhanced" distances, defined with respect to Wick powers of the Gaussian Free Field, which are analogous to rough path metrics. In dimension $3$, two natural generalisations of the OM functional are proved to be degenerate.  Finally, we recover the  $\Phi^4_3$ action, under appropriate regularity conditions, by considering joint small radius-large frequency limits.
\end{abstract}

\msc{Primary: 60H30, Secondary: 28C20, 81T08}
\vspace{2mm}

\color{black}

\section{Introduction}
\subsection{$\Phi_d^4$ theory}\label{subsec:IntroPhi4}
The $\Phi_d^4$ theory in finite volume is concerned with the action functional that is formally defined by
\begin{equation}\label{eq:action}
    S(\phi):=\int_{\mathbb T^d} \left(\frac{1}{4}\phi(x)^4+\frac{m}{2} \phi(x)^2+\frac{1}{2}|\nabla \phi(x)|^2\right) dx,
\end{equation}
where $\phi$ lies in the space $\mathscr{S}'(\T^d)$ of Schwartz distributions on the $d$-dimensional torus $\mathbb T^d.$ 
The theory describes the quartic self-interaction of a single scalar quantum field of mass $m\geq 0$ and also appears as a continuum effective field theory of discrete statistical systems near criticality; see e.g.  \cite{cassandro1995corrections, hairer2018tightness} and references therein for the case of Ising-Kac models in $d=2.$

From the viewpoint of Constructive Quantum Field Theory (CQFT), the corresponding ``Gibbs" probability measure
\begin{equation}\label{eq:phi4-density}
    \mu=``\frac{1}{Z}e^{-S(\phi)}d^\infty \phi",
\end{equation}
where $d^\infty \phi$ is the nonexistent infinite dimensional Lebesgue measure and $Z$ is a normalization constant, is of prime importance. Indeed, $\mu$ represents a simple example of Euclidean QFTs, a class of measures on $\mathscr{S}'(\T^d)$ whose correlation functions satisfy certain properties known as the Osterwalder-Schrader Axioms. Upon verification of these axioms, one obtains, from a Euclidean QFT, a (Bosonic) relativistic QFT, in the sense of Wightman, in $d-$dimensional spacetime. The latter is one of the central aims of CQFT; the reader is referred for example to \cite[Section 1.1]{barashkov2022variational} or \cite[Section 6.1]{glimm2012quantum} and references therein, for a more detailed introduction to the so-called Osterwalder-Schrader reconstruction theorem.

Besides the absence of Lebesgue measure, rigorous constructions of $\mu$ present major mathematical challenges. Indeed,
since $\phi$ is a distribution and $S$ contains polynomial nonlinearities, both \eqref{eq:action} and $\mu,$ are a priori ill-defined. Apart from the simple case of $d=1,$ where one may actually define $\mu$ in a space of functions, such problematic products of distributions are present for any $d\geq 2.$ In essence, the rationale behind the achievements of the CQFT program in the 70's and 80's is to interpret and construct $\mu$ as the limit of appropriately ``renormalised" measures $\{\mu_n\}_{n\in\N}.$ Renormalisation here refers to a systematic framework for replacing divergent quantities, that appear in finite-dimensional approximations of $\mu$, by well behaved counterparts that converge to a finite limit. A detailed account of the relevant literature is far from the scope of this work, we mention however \cite{ nelson1966quartic, feldman1974lambdaphi34} for constructions in $d=2$ and $d=3$ with small coupling strength (i.e. when the quartic term in \eqref{eq:action} is multiplied by a sufficiently small scalar).

In this work, we shall focus exclusively in the case of dimensions $d=1,2,3.$ Modern approaches to CQFT have recently succeeded in constructing the $\Phi^4_3$ measure at all coupling strengths and have been largely related to major advances in the study of singular Stochastic Partial Differential Equations (SPDEs) such as Hairer's novel theory of regularity structures \cite{hairer2014theory} and the paracontrolled calculus of Gubinelli-Imkeller-Perkowski \cite{GIP-paracontrolled}. This relation arises through the program of stochastic quantisation, originating in works of Nelson, Parisi and Wu, which has recently found applications to a wider range of physical theories (once again we refrain from exhaustive references and instead refer the reader to \cite{barashkov2022variational} for a more accurate historical account). In a nutshell, stochastic quantisation aims to define the measure $\mu$ as the invariant measure of an associated Langevin dynamics. In the case of $\Phi^4,$ the latter are governed by the gradient flow SPDE  (also known as the dynamic $\Phi^4$ model) which can be formally written as
\begin{equation}\label{eq:DynamicPhi4Intro}
    \partial_t \Phi =\Delta \Phi-m\Phi-\Phi^3+ ``\infty" \Phi+\xi.
\end{equation}
Here, $\xi$ is space-time white noise on the torus and $ ``\infty"\Phi$ represents a diverging counterterm necessary to renormalise the cube $\Phi^3.$

The complexity of solution theories for \eqref{eq:DynamicPhi4Intro} (and hence of constructions of \eqref{eq:phi4-density}) increases significantly from $d=2$ to $d=3.$ Classical work of Da Prato and Debussche \cite{da-prato-debussche} in $d=2$ shows that meaningful solutions can be constructed after replacing $\xi$ by a smooth approximation $\xi_n$ and reinterpreting $\Phi_n^3$ as a ``Wick-renormalised" cube $\Phi_n^{:3:}$ (see Section \ref{subsec:Wick} for precise definitions). After taking $n\to\infty,$ renormalised solutions $\Phi_n$ converge in probability to a limiting (random) Schwartz distribution $\Phi$ which is then defined to be the solution of \eqref{eq:DynamicPhi4Intro}. Wick renormalisation is also sufficient to construct solutions of the two-dimensional $P(\Phi)_2$ model \cite{Pphi2-book} which generalises $\Phi^4_2$ to higher order interactions by replacing the quartic term in \eqref{eq:action} with an even degree polynomial $P.$ Moreover, the corresponding measures can be expressed in terms of the Gaussian Free Field (GFF) $\mu_0$, defined as the (centered, Gaussian) invariant measure of the stochastic heat equation
\begin{equation}\label{eq:she-background}
    \partial_t\Phi_0 =(\Delta-m)\Phi_0+\xi.
\end{equation}

The case $d=3$ is more ``degenerate" in that it requires an additional ``mass renormalisation." This amounts to adding a logarithmically divergent counterterm in the quadratic term of \eqref{eq:action}. Local solutions of \eqref{eq:DynamicPhi4Intro} are then available within the frameworks of the aforementioned singular SPDE theories (see e.g. \cite{MH-phi4-rs} for the regularity structure approach) and the associated invariant measures are mainly understood as weak limit points of renormalised measures $\mu_n.$ Beyond this additional degeneracy, the $\Phi^4_3$ measure $\mu$ and GFF $\mu_0$ are mutually singular \cite{barashkov2021varphi,Hairer-phi4-note, hairer2024singularity} and hence the latter can no longer serve as a reference measure.

 In fact, even though these measures have been rigorously constructed through powerful limiting arguments, more intrinsic descriptions are desirable but typically absent from the literature. In this direction, Barashkov and Gubinelli proved a novel variational formula for the Laplace transform of the $\Phi^4_3$ and obtained a new construction based on Girsanov's theorem \cite{barashkov2020variational, barashkov2021varphi}. Nevertheless, it is, at the moment of writing, unclear whether $\mu$ possesses the desired ``density" $\frac{1}{Z}e^{-S(\phi)}$ in some appropriate sense.
 
 Our goal is to investigate this question for the $\Phi^4_d$ $(d=1,2,3),$ as well as $P(\Phi)_2,$ measures through the lens of the Onsager-Machlup (OM) functional, a notion of density that is well-defined for measures in infinite dimensional spaces. In particular, we compute (generalised versions of) the OM functionals for these measures and check whether they agree with this informal notion of density.  To the best of our knowledge, the question of determining OM functionals for these measures is considered here for the first time.

\subsection{Onsager-Machlup functional}\label{subsec:OMintro}

The Onsager-Machlup (OM) functional originates from statistical mechanics \cite{OM-Original} and is defined in analogy to Lebesgue's differentiation theorem. In absence of a translation invariant measure, the key idea is to compare a probability measure to translations of itself.

More precisely, we have that if $(X,d)$ is a metric space, $\mu$ a Borel probability measure and the limit 
\begin{equation}\label{eq:OMdefinition}
    \lim_{r\to 0^+}\frac{\mu(B_r(z_1))}{\mu(B_r(z_2))}=\exp\left(\operatorname{OM}(z_2)-\operatorname{OM}(z_1)\right),
\end{equation} 
where $B_r(z_i), i=1,2,$ is a  metric ball of radius $r$ centered at $z_i,$ 
exists for all  $z_1,z_2\in Z\subseteq X$, then $\operatorname{OM}$ is called the \emph{Onsager-Machlup functional} of $\mu$ on $Z$. It is straightforward to check that if $X=\R^d$ and $\mu$ has either a density or a mass function $f$, then its OM functional is given by the negative logarithm of $f$; this is precisely Lebesgue's differentiation theorem or continuity of measure. 

More importantly, OM functionals exist in infinite dimensions as well. For example, any measure that is equivalent to a Gaussian measure on a Banach space admits a (non-trivial) Onsager-Machlup functional, provided that the corresponding density satisfies certain continuity conditions, see \cite{Dashti}. This includes the case of diffusion laws, studied e.g. in \cite{Durr-Bach} (for $X=C([0,T];\R^d), Z=C^2$) and later extended, among others, by Zeitouni \cite{zeitouni1989onsager} to $Z=C^{1+\alpha}$ and Lyons-Zeitouni \cite{lyons1999conditional} to a large class of norms on path space. Beyond the setting of diffusions, a large number of works have studied the OM functional for different measures such as e.g. laws of non-anticipative SDEs \cite{carmona1992traces} or more recently SLE loop measures \cite{carfagnini2024onsager} to name only a few. In Section \ref{subsec:OMbackground} we provide further background and describe a few additional properties of OM functionals in connection to this work.

\subsection{Contributions}

One of the subtleties in computing OM functionals is their sensitivity to the choice of metric. Indeed, even equivalent metrics in $\R^d$ can lead to different OM functionals, see \cite[Example B.4.]{OM-Gamma-1}. Starting from this observation, we propose reasonable topologies for which the OM functional of the $\Phi^4-$measure coincides with the action $S$ \eqref{eq:action} in dimensions $1$ and $2$ but fails to exist in dimension $3$.  

Without loss of generality and throughout the rest of this work we shall consider the massless case $m=0.$ All the subsequent results continue to hold when $m>0$ since a mass term can always be inserted at the level of the Cameron-Martin norm of the GFF measure. Our contributions are summarized below and all the necessary notation, preliminaries as well as additional background on OM functionals are contained in Section \ref{sec:NotationBackground}.

\textbf{1.} In the case $d=1,$  where no renormalisation is required, we show that the OM functional of $\Phi^4_1$, defined in $X=C^{\alpha}(\mathbb{T})$ for any $\alpha\in[0, 1/2)$ exists and coincides with the action functional \eqref{eq:action} (when the latter is extended to $C^\alpha$  by the value $\infty$ on $C^\alpha\setminus H^1$). More precisely, we show:     \begin{theorem}\label{thm:1d-warmup}
   Let $\alpha\in[0, 1/2),$ $\mu$ be the $\Phi_1^4$ measure supported on $C^{\alpha}(\T)$. For any $z_1,z_2\in H_0^1(\mathbb T^1)$ we have
   \begin{equation}\label{thm:main-1d}
        \lim_{r\to 0^+}\frac{\mu(B^1_r(z_1))}{\mu(B^1_r(z_2))}=\exp\left(S(z_2)-S(z_1)\right),
   \end{equation}
where $B^1_r(z_i),$ $i=1,2,$ are norm-open balls of radius $r>0$ and centered at $z_i.$ In particular,
\begin{equation*}
        \operatorname{OM}(z)\equiv S(z)=
        \begin{cases}
            \int_{\mathbb T^1} \frac{1}{4}z^4(x) dx+\frac{1}{2}\|z\|_{H_0^1}^2 &\text{ if } z\in H_0^1(\mathbb T^1)\\
            \infty &\text{ else.}
        \end{cases}
    \end{equation*}
\end{theorem}
\noindent The proof of Theorem \ref{thm:main-1d} can be found in Section \ref{sec:d=1}.

\textbf{2.} In $d=2,$ we let  $X=C^{-\alpha}, \alpha>0$ be a H\"older-Besov space on which the GFF $\mu_0$ is concentrated and consider the $P(\Phi)_2$ measure on $X$ (which includes the $\Phi^4_2$ measure as a special case). Even though the latter is absolutely continuous with respect to $\mu_0,$ its Wick-renormalised density $Z^{-1}\exp(-:P:(\Phi))$ is not continuous in $\Phi.$ Nevertheless, we are able to reclaim continuity by incorporating higher order terms in the underlying topology; an idea that appears naturally in both rough analysis and in the aforementioned constructions of $\Phi^4_2$. In particular, we consider a generalised OM functional in which the balls $B_r$ \eqref{eq:OMdefinition} are replaced by ``enhanced" measurable sets $\mathcal{B}_r(z)$ chosen so that a) they shrink to a point $z$ as $r\to 0$ and b) allow control over Wick powers $\Phi^{:p:}, p=1,\dots,k-1,$ where $k$ is the degree of $P.$ Then, we are once again able to show that this (generalised) OM functional agrees with the action \eqref{eq:action}:

   \begin{theorem}\label{thm:main-2d}
    Let $r,\kappa>0, \alpha>0,$  $z_i\in C^{\alpha+\kappa}(\mathbb T^2)\cap H_0^1(\mathbb T^2), i=1,2,$ and $P:\R\rightarrow\R$ be a polynomial of even degree $2k\in\N$ with strictly positive leading coefficient. Furthermore, consider the ``enhanced" sets    
    \begin{equation}\label{eq:2d-ball}
    \B^2_r(z_i):=\{\phi\in C^{-\alpha}(\mathbb T^2): \|\phi-z_i\|_{C^{-\alpha}},\|(\phi-z_i)^{:2:}\|_{C^{-\alpha}}<r,\|(\phi-z_i)^{:3:}\|_{C^{-\alpha}}<r\},
    \end{equation}
     \begin{equation}\label{eq:enhanced-ball-2d-P}
        \B^P_r(z):=\{\phi\in C^{-\alpha}(\mathbb T^2): \|\phi-z\|_{C^{-\alpha}}<r,...,\|(\phi-z)^{:2k-1:}\|_{C^{-\alpha}}<r\},
    \end{equation} 
    $r>0.$ The following hold:

    \begin{itemize}
        \item[(1)]  Let $\mu$ be the $\Phi_2^4$ measure and $S$ as in \eqref{eq:action}. Then
    \begin{equation*}
        \lim_{r\to 0^+}\frac{\mu(\B^2_r(z_1))}{\mu(\B^2_r(z_2))}=\exp\left(S(z_2)-S(z_1)\right).
    \end{equation*}
In other words, the OM functional of $\mu$ on $C^{\alpha+\kappa}(\mathbb T^2)\cap H_0^1(\mathbb T^2)$ is 
 given by $S.$
\item[(2)] Let $\mu$ be the $P(\Phi)_2$ measure and $S^P$ the action functional 
\begin{equation}\label{eq:PphiAction}
    S_P(\phi)=\int_{\mathbb T^2}\left(P(\phi(x))+\frac{1}{2}|\nabla \phi(x)|^2\right)dx.
\end{equation}
Then
    \begin{equation*}
        \lim_{r\to 0^+}\frac{\mu(\B^P_r(z_1))}{\mu(\B^P_r(z_2))}=\exp\left(S_P(z_2)-S_P(z_1)\right).
    \end{equation*}
In other words, the OM functional of $\mu$ on $C^{\alpha+\kappa}(\mathbb T^2)\cap H_0^1(\mathbb T^2)$ is given by $S_P.$
    \end{itemize}   
\end{theorem}

The proof of Theorem \ref{thm:main-2d} can be found in Section \ref{sec:d=2}. We highlight here, once again, the rough analysis flavor of Theorem \ref{thm:main-2d}: In the case of differential equations driven by (random) rough signals, one obtains continuity of the It\^o-Lyons solution map by enhancing the driving noise with higher order iterated integrals into a (random) rough path. Similarly, continuity of the  $\Phi^4_2$ and $P(\Phi)_2$ densities is regained by considering the enhanced ``balls" \eqref{eq:2d-ball}, \eqref{eq:enhanced-ball-2d-P}. The role of iterated integrals of paths is played here by Wick powers and this enhancement is necessary to gain control over small ball estimates (see also Remark \ref{rem:support} below).

\textbf{3.} In $d=3$ we are faced with both the mutual singularity of $\mu_0, \mu$ but also the fact that the Wick cube $\Phi_n^{:3:}$ of the Fourier-truncated GFF diverges logarithmically as $n\to\infty.$ In complete analogy to our approach in $d=2,$ we start from a sequence of renormalised measures $\{\mu_n\}_{n\in\N}$ on $X=C^{-1/2-\kappa}(\T^3), \kappa>0,$ converging to the $\Phi^4_3$ measure. Then we consider enhanced sets $\mathcal{B}^3_r(z)$ that allow control over the Wick powers $\Phi_n^{:p:}, p=1,2, 3$ but also incorporate the logarithmic divergence of $\Phi_n^{:3:}.$ In sharp contrast to the previous case, we show that this time the corresponding OM functional is infinite for all smooth $z_1\neq 0$ and $z_2=0$  in $X:$

\begin{theorem}\label{theorem:main-3d-full}
    Let $\kappa>0, z\in C^\infty(\mathbb T^3), \mu$ be the $\Phi^4_3$ measure and consider the sets     \begin{equation}\label{eq:3d-ball}
    \begin{split}
    \mathcal{B}^3_r(z):=\big\{\phi\in C^{-1/2-\kappa}(\mathbb T^3): \forall k\in\N,\; \|(\phi-z)_k\|_{C^{-1/2-\kappa}}<r, &\|(\phi-z)_k^{:2:}\|_{C^{-1-\kappa}}<r,\\
    \|(\phi-z)_k^{:3:}\|_{C^{-3/2-\kappa}}<r\log k\big\}\;,\;r>0.
    \end{split}
\end{equation} If $\mu(\B^3_r(0))>0$ for all $r>0,$ then the Onsager-Machlup functional for the $\Phi_3^4$ measure, with respect to the sets $\mathcal{B}^3_r$ is given, up to an additive constant, by 
    \begin{equation}\label{eq:OM3d}
        \operatorname{OM}(z):=\begin{cases}
            \infty &\text{ if }z\neq 0\\
            0 &\text{ if }z=0.
        \end{cases}
    \end{equation}
    In particular, the corresponding  $\operatorname{OM}$ functional is either degenerate or not well-defined.
\end{theorem}   
    
    The proof of Theorem \ref{theorem:main-3d-full} can be found in Section \ref{sec:d=3}. Despite this degeneracy result, we are still able to recover the action $S$ by considering pairs  $z_1, z_2\neq 0$ that satisfy some further compatibility conditions and considering a joint limiting  regime 
    $n=n(r)\rightarrow\infty$ as $r\to 0$ with $\lim_{r\to 0^+} \log n(r)r=0$ (Proposition \eqref{prop:recover-1}). Finally, we compute a ``generalised" small ball ratio over multiple balls and show that it converges, in the same joint asymptotic regime, to a finite expression involving $S$ for all smooth $z_1, z_2.$ These computations are presented in Section \ref{sec:recover}. Before we move on, let us emphasize here that Theorem \ref{theorem:main-3d-full} does not preclude the existence of a generalised notion of density for $\Phi^4_3;$ rather, it highlights that a natural candidate for the latter is degenerate. A more detailed discussion on our results and potential future directions is postponed to Section \ref{sec:Conclusion}.

\section*{Acknowledgments}
The authors would like to thank Ilya Chevyrev for several helpful discussions.

\section{Notations and preliminaries}\label{sec:NotationBackground}

\begin{notation}\label{notation:background} Throughout this work $\mathbb T^d$ is the $d$-torus for $d=1,2,3$.  For $\alpha\in\R,$ $C^{\alpha}(\mathbb T^d)$ is the H\"older-Besov space of distributions defined as the completion of $C^\infty(\T^d)$ under the norm 
    $$\|f\|_{C^\alpha}:=\sup_{j\geq -1} 2^{j\alpha}\|\Delta_jf\|_{L^\infty(\T^d)},    $$
    where $\{\Delta_jf\}_{j\geq-1}$ are the  Paley-Littlewood projectors of $f$ (see e.g. \cite[Chapter 2]{bahouri2011fourier} for definitions and details). When $\alpha>0$ the latter is isometric to the space of H\"older continuous functions. The duality pairing between a distribution $f\in C^\alpha(\T^d)$ and a smooth function $\phi\in C^\infty(\T^d)$ is denoted by $\langle f,\phi\rangle.$ 
    
   For $\alpha\in\R, f\in C^\alpha(\T^d)$ and a multi-index $k=(k_1,\dots,k_d)\in\Z^d,$  $\hat{f}_k$ is the $k-$th Fourier coefficient of $f.$ For $N\in\N$ we also use the notation $f_N=\sum_{|k|\leq N}\hat{f}_ke_k,$ where $\{e_k\}_{k\in\Z^d}\subset L^2(\T^d)$ is an orthonormal Fourier basis and $|k|:=k_1+\ldots+k_d.$ $H^1_0(\T^d)$ is the Sobolev space of mean-zero $H^1$ functions $f$ endowed with the norm $\|f\|_{H^1_0}=\|\nabla f\|_{L^2}.$
   
    For any $\kappa>0,$ $\mu_0$ is the law Gaussian Free Field (GFF) and $\mu$ is the $\Phi_d^4$ measure on $C^{1-d/2-\kappa}(\T^d).$ The former is defined as the centered Gaussian measure with covariance operator $Q=(-\Delta)^{-1}$ and Cameron-Martin space $H^1_0(\T^d).$
    
    A smooth function $f\in C(\T^d)\equiv C^0(\T^d)$ is said to be of finite spectral support if there exists $N\in\N$ such that $\hat{f}_k=0$ for all $k\in\Z^d$ with $|k|\geq N.$
    
     The continuous dual of a Banach space $X$ is denoted by $X^*.$ For a metric space $X, B_r(z)$ is the open ball of radius $r>0$ centered at $z\in X.$
     Finally, for $x, y\in\R$ we write $x\lesssim y$ to denote that $x\leq C y$ up to an (unimportant) constant $C>0.$
\end{notation} 

\begin{lemma}[Duality estimate in H\"older-Besov spaces]\label{lem:dualityBound} For any $\alpha\in\R, \delta>0$ and distributions $f\in C^{\alpha}(\mathbb T^d), g\in C^{-\alpha+\delta}(\mathbb T^d)$ we have $\langle f,g\rangle\lesssim \|f\|_{C^\alpha}\|g\|_{C^{-\alpha+\delta}}.$    
\end{lemma}
\begin{proof} By \cite[Proposition 2.76]{bahouri2011fourier}, compactness of the torus and the continuous Besov space embeddings $C^{-\alpha+\delta}(\mathbb T^d):=B_{1,\infty}^{-\alpha+\delta}(\mathbb T^d)\hookrightarrow B_{1,\infty}^{-\alpha+\delta}\hookrightarrow B_{1,1}^{-\alpha}(\mathbb T^d)$ (see e.g. \cite[Corollary 2.96]{bahouri2011fourier}) we get
 $\langle f,g\rangle\lesssim \|f\|_{C^\alpha}\|g\|_{B^{-\alpha}_{1,1}}\lesssim \|f\|_{C^\alpha}\|g\|_{C^{-\alpha+\delta}}.$ \end{proof}

\begin{lemma}[Uniform convergence of Fourier series]\label{lem:FourierConvergence} Let $d,k\in\N$ and $f\in C^{k+d/2}(\T^d).$ Then $\|f_n-f\|_{C^k(\T^d)}\longrightarrow 0$ as $n\to\infty.$
\end{lemma}
\begin{proof}\cite[Theorem 3.3.16]{grafakos2014fourier}.   
\end{proof}

\subsection{Wick powers and renormalised approximations}\label{subsec:Wick}
As discussed in Section \ref{subsec:IntroPhi4}, in two dimensions, Wick renormalisation can be used to make sense of equation \eqref{eq:DynamicPhi4Intro} and rigorously construct the invariant measure. Wick powers of the GFF (in $d=2)$ are defined as follows:
\begin{notation}\label{notation:Wick}
    Let $\alpha>0$ and $\Phi$ be a sample from the GFF, i.e. a $C^{-\alpha}(\T^2)-$valued, $\mu_0-$distributed random element.
    With $c_n:=\ex_{\mu_0}[\Phi_n^2], p\in\N,$ we denote by $\Phi_n^{:p:}$ the $p$-th Wick power of $\Phi_n$, defined by $$\Phi_n^{:p:}(x)=H_{p}(\Phi_n(x), c_n),\;\;x\in\T^2$$ where, for some $c>0$ and all $x\in\R,$ $H_p(y, c)=-c^n\exp(y^2/2c)\tfrac{d}{dy}\exp(-y^2/2c)$ is the $p-$th Hermite polynomial with variance $c.$   
    It is well known (see e.g. \cite{da-prato-debussche}) that $\Phi_n^{:p:}$ converges in $L^p(\mu_0)$ to a limiting random distribution $\Phi^{:p:}$ in $C^{-\alpha}(\T^2).$ The latter is called the $p-$th Wick power of $\Phi.$
\end{notation}
With this notation in hand one can express the $\Phi_4^2$ measure (viewed as the unique invariant measure to \eqref{eq:DynamicPhi4Intro}, see e.g.  \cite{da-prato-debussche}) with respect to the GFF as
    \begin{equation}\label{eq:inv-measure-background-2d}
        \mu(d\phi)=\frac{1}{Z}\exp\left(-\frac{1}{4}\langle \phi^{:4:},1\rangle\right)\mu_0(d\phi),
    \end{equation}
    where $Z$ is a normalizing factor. An analogous expression $$\mu(d\phi)=\frac{1}{Z}\exp\left(-\langle :P:(\phi),1\rangle\right)\mu_0(d\phi)$$
holds true for the  $P(\Phi)_2$ measure, where for some $k\in\N$ and $y\in\R,$ $P(y)=\sum_{j=0}^{2k}a_j y^{j}$ with $a_{2k}>0$ and $:P:(\Phi)=\sum_{j=0}^{2k}a_j \Phi^{:j:}.$

\begin{remark}\label{rem:support}
The map $\Phi\mapsto\Phi^{:p:}$ is only measurable and can be described in terms of iterated It\^o integrals. Thus, as typically seen in rough analysis, control over $\Phi$ provides little control over $\Phi^{:p:}$. Indeed, using techniques in \cite{Friz-support-sspde,Hairer-support-sspde} one can show that $(\Phi, \Phi^{:2:})$ has full support in $(C^{-\alpha}(\mathbb T^2))^2$. For this reason, and in order to maintain control over small balls, we introduce the enhanced balls \eqref{eq:2d-ball} for the study of the (generalised) OM functional. 
\end{remark}

When $d=3$ the $\Phi_3^4$ measure no longer admits a simple expression like in equation \eqref{eq:inv-measure-background-2d}. Instead, it is obtained as the weak 
limit in $C^{-1/2-\kappa}(\T^3)$ of the measures
\begin{equation}\label{eq:mu_nMeasures3D}
    \mu_n(d\phi)=\frac{1}{Z_n}\exp\left(-\frac{1}{4}\int_{\mathbb T^3} \big(\phi_n^{:4:}(x)-C_n \phi_n^2(x)\big)dx\right)\mu_0(d\phi),
\end{equation}
where $C_n<0$ is a sequence of constants with $C_n=O(\log n)$. These renormalised measures will be our starting point for the proof of Theorem \ref{theorem:main-3d-full}. 
Moreover, this measure does not have a density with respect to the GFF (see e.g. \cite{barashkov2021varphi,Hairer-phi4-note, hairer2024singularity}), and further, is mutually singular with respect to every nonzero smooth shift \cite{Hairer-phi4-note}. 
\subsection{Background on OM functionals}\label{subsec:OMbackground}
As mentioned in Section \ref{subsec:OMintro}, OM functionals generalise the notion of a probability density function to settings where a well-behaved reference measure, such as Lebesgue measure on $\R^d$, is not available.

\begin{definition}
    Let $(X,d)$ be a metric space. Let $\mu$ be a Borel probability measure on $(X,d)$. If for a subset $Z\subset X$ there is a function $\operatorname{OM}:Z\to \mathbb R$ so that for all $z_1,z_2\in Z$ we have
    \begin{equation}\label{eq:om-def}
        \lim_{r\to 0^+}\frac{\mu(B_r(z_1))}{\mu(B_r(z_2))}=\exp(\operatorname{OM}(z_2)-\operatorname{OM}(z_1)),
    \end{equation}
    then $\operatorname{OM}$ is called the Onsager-Machlup (OM) functional of $\mu$ on $Z$. 
\end{definition}
\begin{remark} \begin{enumerate}
    \item[(1)] It is often the case that $Z$ is much smaller than $\operatorname{supp}(\mu)$; see e.g. Proposition \ref{prop:GM} below.
    \item[(2)]    $\operatorname{OM}$ is only defined up to an additive constant. If $\operatorname{OM}$ satisfies \eqref{eq:om-def} then $\widetilde{\operatorname{OM}}:=\operatorname{OM}+C$ also satisfies \eqref{eq:om-def} for any $C\in \mathbb R$. This comes from the fact that there is no natural ``normalization" of OM functionals. In finite dimensions, we choose $C$ so that $\int_{\mathbb R^d} e^{-\operatorname{OM}(x)-C}dx=1.$ 
\end{enumerate} 
\end{remark}

An important class of examples is given by Gibbs-type measures on Banach spaces that are absolutely continuous with respect to Gaussian measures:
\begin{proposition}\label{prop:GM}\cite[Theorem 3.2]{Dashti}
    Let $\mu_0$ be a centered Gaussian measure on Banach space $\mathcal B$ with Cameron-Martin norm $\|\cdot\|_\mu$ and space $\mathcal H_{\mu_0}$. Let $\mu=\frac{1}{Z}e^{-V}\mu_0$ where $V:\mathcal B\rightarrow\R$ is locally Lipschitz. Then
    \begin{equation*}
        \operatorname{OM}(z)=\begin{cases}
            V(z)+\frac{1}{2}\|z\|_{\mu_0}^2 &\text{ if }z\in \mathcal H_{\mu_0}\\
            \infty&\text{ else.}
        \end{cases}
    \end{equation*}
\end{proposition}

\begin{remark}
The local Lipschitz condition on $V$ can be reduced to local uniform continuity, with ``local" referring to bounded sets, without much issue. However, in sharp contrast to finite dimensions (where the above is true when $V$ is only measurable), local uniform continuity requirements on $V$ are essential in infinite dimensions.  This concern has practical importance to the proof of Theorem \ref{thm:main-2d} where the balls $\mathcal{B}^2_r(z)$ are chosen precisely to guarantee that the map $ \mathcal{B}^2_r(z)\ni \phi\longmapsto:P:(\phi+z)\in  C^{-\alpha}$ has enough continuity properties to ensure a well-defined OM functional.
\end{remark}

OM functionals are typically used as Lagrangians that determine most likely paths of diffusion processes \cite{Durr-Bach}. More generally, minimizers of $\operatorname{OM}$ are typically defined to be the \textit{modes} of an underlying measure $\mu$ on an infinite dimensional space. There has been a great deal of study on the notion of modes in infinite dimensions. For example, in the terminology of \cite{Different-modes}, minimizers of OM are called weak modes. A similar role is often played by Freidlin-Wentzell (FW) rate functions in the context of Large Deviation Principles (LDPs). We conclude this section with a short comparison of the two.

\begin{remark}[Comparison to FW rate functions] Despite their similar interpretations, OM functionals are different from FW rate functions. A simple yet central example is provided by the diffusion process $x$ solving $dx_t=f(x_t)dt+\epsilon dW_t,$ where $\epsilon>0, f\in C^2(\R), W$ is a one-dimensional Brownian motion. In this case, the OM functional of the law $\mu^\epsilon$ of $x$ on $C[0,T]$ and the small-noise FW rate function $I_{FW}$ satisfy \begin{equation*}        \operatorname{OM}^\epsilon(z)=
            \frac{1}{2\epsilon^2}\int_0^T (f(z(t))-z'(t))^2 dt+\frac{1}{2}\int_0^T f'(z(t))dt=\frac{1}{\epsilon^2}I_{FW}(z)+\frac{1}{2}\int_0^T f'(z(t))dt,
    \end{equation*}
if $z\in H^1[0,T]$ (equality holds trivially otherwise since both sides are equal to $\infty$). The last term is an It\^o Stratonovich correction and is intimately related to the discontinuity of SDE solution maps; we refer to \cite{Self-FWOM} for more details. Besides the latter, FW rate functions serve as asymptotic ``Lagrangians" when e.g. some parameter, such as temperature/noise intensity of the measure vanishes. Lastly, both functionals play important roles in path sampling applications of SDEs or SPDEs see e.g. \cite{gasteratos2024importance,hairer2007analysis} and references therein. 
\end{remark}

\section{Onsager-Machlup for $\Phi_1^4$}\label{sec:d=1}
As a warm-up, we compute here the  OM functional in $d=1$ when renormalisation is not necessary. 

    \begin{proof}[Proof of Theorem \ref{thm:1d-warmup}]
        Recall that $\mu_0$ is the law of the (massless) GFF (Brownian bridge) on $\mathbb T^1$, the measure $\mu$ has density $\frac{d\mu}{d\mu_0}(\phi)\propto \exp\left(-\frac{1}{4}\int_{\mathbb T^1} \phi^4(x) dx\right)$, and that the Cameron-Martin norm associated to the Gaussian free field is the $H_0^1(\mathbb T^1)$ norm. Therefore, the Cameron-Martin theorem implies that there are some $z_1^\ast, z_2^\ast\in (C^{\alpha}(\mathbb T^2))^\ast$ or $(L^\infty(\mathbb T^1))^\ast$ so that
        \begin{align*}
             \lim_{r\to 0^+}\frac{\mu(B^1_r(z_1))}{\mu(B^1_r(z_2))}&= \lim_{r\to 0^+}\frac{\int_{B^1_r(0)}\exp\left(-\frac{1}{4}\int_{\mathbb T^1} (\phi-z_1)^4(x)dx-z_1^\ast(\phi)-\frac{1}{2}\|z_1\|_{H_0^1}^2\right)\mu_0(d\phi)}{\int_{B^1_r(0)}\exp\left(-\frac{1}{4}\int_{\mathbb T^1} (\phi-z_2)^4(x)dx-z_2^\ast(\phi)-\frac{1}{2}\|z_2\|_{H_0^1}^2\right)\mu_0(d\phi)}.
        \end{align*}
        Applying the standard binomial expansion to $(\phi-z_i)^4$ for $i=1,2$ we note that if $\phi$ is small, then $\phi^2,\phi^3$ are small. Moreover, by the Sobolev embedding $H_0^1(\T)\hookrightarrow C^\alpha(\T)$ and compactness of the domain we have that $z^m\in C^\alpha$ for $m=1,2,3.$ Using that on $B^1_r(0)$ we have $|z_i^\ast(\phi)|\lesssim r$ for $\|\cdot\|=\|\cdot\|_{C^\alpha}$ or $\|\cdot\|_{L^\infty}$ we obtain  that $OM(z)=S(z)$ for all $z\in H_0^1(\T).$ 
    \end{proof}
Therefore, the OM functional in one dimension is the $\Phi_1^4$ action. 
\begin{remark}
    In the proof of Theorem \ref{thm:1d-warmup} we crucially use that $\|\phi\|<r$ implies that $\|\phi^2\|, \|\phi^3\|\lesssim r$ for $\|\cdot\|=\|\cdot\|_{C^\alpha}$ or $\|\cdot\|_{L^\infty}$. This is false in $2$ dimensions, and the reason why we introduce more complicated balls in Theorem \ref{thm:main-2d}.
\end{remark}
\section{Onsager-Machlup for $\Phi_2^4$}\label{sec:d=2}
We first start with a binomial theorem for Wick powers.
\begin{lemma}\label{lemma:binomial}
    Let $\kappa, \alpha>0,$ $z\in H_0^1(\mathbb T^2)\cap C^{\alpha+\kappa}(\mathbb T^2)$ and $p\in \mathbb N$. Then 
    \begin{equation*}
        \langle(\phi-z)^{:p:},1\rangle=\sum_{m=0}^p C_m^p (-1)^m\langle\phi^{:m:},z^{p-m}\rangle
    \end{equation*}
    $\mu_0-$almost surely, where $\phi^{:0:}:=1$ and $C_m^p$ is the binomial coefficient. 
\end{lemma}
\begin{proof} Let $m\in\{0,\dots, p\}.$ The approximate Wick powers $\phi_n^{:m:}$ converge in probability, as $n\to\infty,$ to $\phi^{:m:}$ on $C^{-\alpha}(\T^2)$ (in fact this convergence holds in $L^p(C^{\gamma}(\T^2))$ for any $p\geq 1, \gamma<0;$ see e.g. \cite[Lemma 3.2]{da-prato-debussche}). Therefore, we can find a subsequence (which will also be indexed by $n$ for simplicity) so that almost surely we have
$\lim_{n\to\infty} \phi_n^{:m:}=\phi^{:m:}.$ 
Indeed, as $z\in H_0^1(\mathbb T^2)$ we have by the Cameron-Martin theorem that $(\phi-z)_n^{:p:}$ converges on a set of probability $1$. 

By the standard binomial theorem for Wick powers (e.g. \cite[Corollary 3.4]{da-prato-debussche}) we have that
$$\langle(\phi-z)_n^{:p:},1\rangle=\left\langle\sum_{m=0}^p C_m^p(-1)^m \phi_n^{:m:}z_n^{p-m},1\right\rangle=\left\langle\sum_{m=0}^p C_m^p(-1)^m \phi_n^{:m:},z_n^{p-m}\right\rangle.
$$
As $z\in C^{\alpha+\kappa}(\mathbb T^2)$ and, for $\alpha>0,$ the latter coincides with the space of $\alpha-$H\"older functions,  we have that $z^{p-m}\in C^{\alpha+\kappa}(\mathbb T^2)$ for all $m=0,\dots,p,$ so the action of $\phi^{:m:}$ on $z^{p-m}$ is well defined by Lemma \ref{lem:dualityBound}. Taking limits concludes.
\end{proof}

\noindent We are now ready to prove Theorem \ref{thm:main-2d}(1).
\begin{proof}[Proof of Theorem \ref{thm:main-2d}(1)]
Without loss of generality, let $z_1=z$ and let $z_2=0$. Indeed, we may always write 
$$\frac{\mu(\B^2_r(z_1))}{\mu(\B^2_r(z_2))}=\frac{\mu(\B^2_r(z_1))}{\mu(\B^2_r(0))}\frac{\mu(\B^2_r(0))}{\mu(\B^2_r(z_2))},$$
where
$$\frac{\mu(\B^2_r(z))}{\mu(\B^2_r(0))}=\frac{\int_{\B^2_r(z)}\exp\left(-\frac{1}{4}\langle \phi^{:4:},1\rangle\right)d\mu_0(\phi)}{\int_{\B^2_r(0)}\exp\left(-\frac{1}{4}\langle \phi^{:4:},1\rangle\right)d\mu_0(\phi)}.$$
As $z\in H_0^1(\mathbb T^2)$ we may apply Cameron-Martin theorem to get that 
\begin{equation*}
    \begin{aligned}
        \frac{\mu(\B^2_r(z))}{\mu(\B^2_r(0))}&=\frac{\int_{\B^2_r(0)}\exp\left(-\frac{1}{4}\langle (\phi-z)^{:4:},1\rangle-z^\ast(\phi)-\frac{1}{2}\|z\|_{H_0^1}^2\right)d\mu_0(\phi)}{\int_{\B^2_r(0)}\exp\left(-\frac{1}{4}\langle \phi^{:4:},1\rangle\right)d\mu_0(\phi)}\\&     =\frac{\int_{\B^2_r(0)}\exp\left(-\frac{1}{4}\sum_{m=0}^4 C_m^4(-1)^m \langle\phi^{:m:},z^{4-m}\rangle-z^\ast(\phi)-\frac{1}{2}\|z\|_{H_0^1}^2\right)d\mu_0(\phi)}{\int_{\B^2_r(0)}\exp\left(-\frac{1}{4}\langle \phi^{:4:},1\rangle\right)d\mu_0(\phi)},
    \end{aligned}
\end{equation*}
where we used the binomial formula, Lemma \ref{lemma:binomial} to expand the numerator.

As $z^\ast$ is a continuous linear functional on $C^{\alpha}(\mathbb T^2)$, and $\phi\in \B^2_r(0)$ we have that $|z^\ast(\phi)|\lesssim r$ and $|\langle \phi, z^3\rangle|\leq r\|z^3\|_{C^{\alpha+\kappa}},$ $|\langle \phi^{:2:}, z^2\rangle|\lesssim r\|z^2\|_{C^{\alpha+\kappa}},$ and $|\langle \phi^{:3:}, z\rangle|\lesssim r\|z\|_{C^{\alpha+\kappa}}$ where we used Lemma \ref{lem:dualityBound} and the fact that $z^m\in C^{\alpha+\kappa}(\T^2)$ for all $m=1,2, 3$ by compactness of the torus.
Therefore, there are constants $C_1(r),C_2(r)$ with $\lim_{r\to 0} C_i(r)=1$ for $i=1,2$ so that
\begin{align*}
    \frac{\mu(\B^2_r(z))}{\mu(\B^2_r(0))}&\leq C_2(r) \exp\left(-\frac{1}{4}\langle 1, z^4\rangle-\frac{1}{2}\|z\|_{H_0^1}^2\right)\frac{\int_{\B^2_r(0)}\exp\left(-S(\phi) \right)\mu_0(d\phi)}{\int_{\B^2_r(0)}\exp\left(-S(\phi)\right)\mu_0(d\phi)}\\&
    =C_2(r) \exp\left(-\frac{1}{4}\langle 1, z^4\rangle-\frac{1}{2}\|z\|_{H_0^1}^2\right)
\end{align*}
and an an analogous lower bound holds with $C_2$ replaced by  $C_1.$
Sending $r\to 0$ concludes the proof.
\end{proof}

\begin{remark}
    Theorem \ref{thm:main-2d}(2) follows in exactly the same way. The only difference is that if $P$ is a degree $2k$ polynomial, then we need to define the ``enhanced ball"
    \begin{equation}
        \B^P_r(z):=\{\phi\in C^{-\alpha}(\mathbb T^2): \|\phi-z\|_{C^{-\alpha}}<r,...,\|(\phi-z)^{:2k-1:}\|_{C^{-\alpha}}<r\},
    \end{equation}
    where $z\in C^{\alpha+\kappa}(\mathbb T^2)\cap H_0^1(\mathbb T^2).$ The terms $\langle z^{j}, \phi^{:2k-j:}\rangle, j=1,\dots 2k-1$ can once again be estimated by Lemma \ref{lem:dualityBound} and the Wick powers are of order $r$ when $\phi\in \B^P_r(z).$
\end{remark}

\begin{remark}
    We again emphasize that $\phi$ is small does not imply that $\phi^{:2:}$, $\phi^{:3:}$ are small. In order to get control over the terms involving these Wick powers, we include them in the ``balls" $\B_r^2.$ 
\end{remark}

\section{Onsager-Machlup for $\Phi_3^4$}\label{sec:d=3}
As we have already mentioned, even though the truncated Wick square of the GFF converges in probability to a well-defined random distribution in $C^{-1-\kappa}(\T^3),$  the Wick cube $\phi_k^{:3:}$ diverges logarithmically in $k$ (see e.g. \cite[pp.68, equation (2.69)]{barashkov2022variational}) when $d=3.$ Due to this divergence, we shall fix $\kappa>0, z\in H_0^1(\T^3)$ and consider an OM functional for $\Phi^4_3$ with respect to the sets
\begin{equation*}
\begin{split}
    \mathcal{B}^3_r(z):=\big\{\phi\in C^{-1/2-\kappa}(\mathbb T^3): \forall k\in\N,\; \|(\phi-z)_k\|_{C^{-1/2-\kappa}}<r, &\|(\phi-z)_k^{:2:}\|_{C^{-1-\kappa}}<r,\\
    \|(\phi-z)_k^{:3:}\|_{C^{-3/2-\kappa}}<r\log k\big\}.
    \end{split}
\end{equation*}
These enhanced sets are natural since they are chosen in direct analogy with the sets \eqref{eq:2d-ball} from Theorem \ref{thm:main-2d}. There are, however, different choices of sets that lead to the same degeneracy results for the corresponding OM functionals. Instead of the sets defined in \eqref{eq:3d-ball}, we can also consider the ``fully renormalised ball": 
\begin{equation}\label{eq:3d-ball-full-renorm}
\begin{split}
   \hat{\B}^3_r(z):=\{\phi\in C^{-1/2-\kappa}(\mathbb T^3): \|(\phi-z)_k\|_{C^{-1/2-\kappa}}<r, \|(\phi-z)_k^{:2:}\|_{C^{-1-\kappa}}<r&,\\
    \|(\phi-z)_k^{:3:}-C_k(\phi-z)\|_{C^{-3/2-\kappa}}<r\log k\text{ for all }k\in \mathbb N\},&
    \end{split}
\end{equation}
with $C_k$ as in \eqref{eq:mu_nMeasures3D}.

In contrast to $d=2$ it turns out that both of the  ``enhanced balls" $\mathcal{B}_r^3(z)$ or $\hat{\B}_r^3(z),$ around any nonzero smooth function $z$, are $\mu-$null sets if $r$ is small enough. This is the main reason why the OM from Theorem \ref{theorem:main-3d-full} is degenerate. 

\begin{proposition}\label{prop:3d-meas0}
    Let $0\neq z\in C^\infty(\mathbb T^3),$ $\{\B^3_r(z)\}_{r>0}$ as in \eqref{eq:3d-ball} and $\mu$ denote the $\Phi_3^4$ measure on $\T^3.$ There exists $r_0=r_0(z)>0$  such that $\mu(\B^3_{r_0}(z))=0$ and $\mu(\hat{\B}_{r_0}^3(z))=0.$
\end{proposition}
\begin{proof}
First, we show that $\B^3_r(z)$ is contained in $\hat{\B}^3_{(1+C)r}(z)$ for some $C>0$, so proving that there is some $r_0>0$ with $\mu(\hat{\B}_{r_0}^3(z))=0$ is sufficient.

Indeed, let $\phi \in \B^3_r(z)$. As $C_k$ diverges logrithmically, it then follows that there is a $C>0$ so that
\begin{align*}
     \|(\phi-z)_k^{:3:}-C_k(\phi-z)\|_{C^{-3/2-\kappa}}&\leq  \|(\phi-z)_k^{:3:}\|_{C^{-3/2-\kappa}}+\|C_k(\phi-z)\|_{C^{-3/2-\kappa}}\\
     &\leq r \log k+|C_k|\|(\phi-z)\|_{C^{-3/2-\kappa}}\\
     &\leq r \log k + C \log k \|(\phi-z)\|_{C^{-3/2-\kappa}}\\
     &\leq r \log k + C \log k \|(\phi-z)\|_{C^{-1/2-\kappa}}\\
     &\leq (1+C)r \log k.
\end{align*}
Therefore $\phi \in \hat{\B}_{(1+C)r}^3.$

Now all that is left is to show that there is an $r_0>0$ so that $\mu(\hat{\B}^3_{r_0}(x))=0.$ Using Lemma \ref{lemma:binomial}, for any smooth $\psi$ we have that 
    \begin{align*}
        \langle (\phi-z)_k^{:3:}-C_k(\phi-z), \psi \rangle&=\langle \phi_k^{:3:}-3\phi_k^{:2:}z_k+3\phi_kz_k^2-z_k^3-C_k\phi+C_kz, \psi \rangle\\
        &=\langle  \phi_k^{:3:}-C_k \phi, \psi\rangle-3\langle \phi_k^{:2:}z_k, \psi\rangle +3\langle \phi_kz_k^2, \psi \rangle\\
        &~~-\langle z_k^3,\psi\rangle+C_k\langle z,\psi\rangle.
    \end{align*}
    In view of \cite[Lemmas 4.2-4.5]{Hairer-phi4-note} (with the exponent $\gamma$ replaced by $1$)
    $$\lim_{k\to\infty}\frac{1}{\log k}\langle  \phi_k^{:3:}-C_k \phi, \psi\rangle=0$$
    $\mu-$almost surely. Furthermore, $\phi_k^{:2:}$ and $\phi_k$ converge $\mu-$almost surely to well-defined distributions (see e.g. \cite[Lemma 4.4]{Hairer-phi4-note}) and $z_k$ converges to $z$ in $C^\infty$ (per Lemma \ref{lem:FourierConvergence}). Since $C_k\lesssim\log k$ it follows that
    $$\lim_{k\to\infty}\frac{1}{\log k} \langle (\phi-z)_k^{:3:}-C_k(\phi-z), \psi \rangle=K \langle z,\psi\rangle$$
    for some $K>0$. From these observations we deduce that
    $$ \mu\big(   \hat{\B}^3_r(z)\big)\leq \mu\bigg( \bigg\{  \limsup_{k\to\infty}\frac{1}{\log k}\|(\phi-z)_k^{:3:}-C_k(\phi-z)\|_{C^{-3/2-\kappa}}<r    \bigg\} \bigg)                      $$
    and $$   \limsup_{k\to\infty}\frac{1}{\log k}\|(\phi-z)_k^{:3:}-C_k(\phi-z)\|_{C^{-3/2-\kappa}}\geq K |\langle z,\psi\rangle |     $$
    $\mu-$almost surely
    for any smooth $\psi$ with $\|\psi\|_{C^{3/2+\kappa}}=1.$  Choosing $\psi$ so that $\langle z,\psi\rangle \neq 0$ and $r<K|\langle z,\psi\rangle|$ yields the desired conclusion.
\end{proof}
Since $\mu(\B^3_r(z))=0$ for $z$ smooth and $r$ small enough, the OM functional for these balls is not well-defined. In particular,  Proposition \ref{prop:3d-meas0} immediately implies the proof of our degeneracy result: 

\begin{proof}[Proof of Theorem \ref{theorem:main-3d-full}] Let $z\in C^\infty(\mathbb T^3).$ On the one hand, if $\B^3_r(0)$ has positive $\mu-$measure for all $r>0$ then there exists, by Proposition \ref{prop:3d-meas0}, a $r_0=r_0(z)>0$ such that for all $r\leq r_0$
 \begin{equation*}
     \frac{\mu(\B^3_r(z))}{\mu(\B^3_r(0))}=\begin{cases}
            0 &\text{ if }z\neq 0\\
            1 &\text{ if }z=0.
        \end{cases}
 \end{equation*}
     Taking $r\to 0^+$ immediately yields \eqref{eq:OM3d} (recall \eqref{eq:OMdefinition}). On the other hand, if there exists $r'>0$ for which $\mu(\B^3_{r'}(0))=0$ then, by monotonicity, we have $\mu(\B^3_{r}(0))=0$ for all $r\leq r'.$ Hence, neither the above ratio nor its limit as $r\to 0^+$ are well-defined. The proof is complete.    
\end{proof}

\noindent We emphasize that, regardless of whether $\mu(\B^3_r(0))=0$ for $r$ small enough or $ \mu(\B^3_r(0))>0$ for all $r$, the OM functional does not agree with the $\Phi_3^4$ action. This is not unreasonable, especially since $\Phi_3^4$ measure is mutually singular with respect to its pushforward by any smooth shift \cite{Hairer-phi4-note, hairer2024singularity}.

\section{Restoring finiteness in $3$ dimensions}\label{sec:recover}
Although the OM functional (in the sense of Theorem \ref{theorem:main-3d-full}) is degenerate, we can recover the $\Phi_3^4$ action for certain pairs of smooth functions $z_1,z_2$ by tying $n,r$ together. 
\begin{proposition}\label{prop:recover-1}
    Let $z_1, z_2\in C^\infty(\mathbb T^3)$ be so that $$\lim_{n\to\infty} \log n \left(\int_{\mathbb T^3}z_{1,n}^2(x)dx-\int_{\mathbb T^3} z_{2,n}^2(x) dx\right)=0.$$
    Let $n(r)$ be a sequence of natural numbers depending on $r$ so that $\lim_{r\to 0^+} \log n(r)r=0.$ Then 
    $$\lim_{r\to 0^+} \frac{\mu_{n(r)}(\mathcal{B}^3_r(z_1))}{\mu_{n(r)}(\mathcal{B}^3_r(z_2))}=\exp(S(z_2)-S(z_1)).$$
\end{proposition}
\begin{proof}
Let $z_i\in C^\infty(\mathbb T^3)$ for $i=1,2$. Then in view of \eqref{eq:mu_nMeasures3D} we have that
$$\mu_n(\B^3_r(z_i))=\frac{1}{Z_{n}}\int_{\B^3_r(z_i)}\exp\left(-\frac{1}{4}V_n(\phi_n)\right)\mu_0(d\phi),$$
where 
$$V_n(\phi_n):=\int_{\mathbb T^3} \big(\phi_n^{:4:}(x)- C_n \phi_n^{:2:}(x)\big)dx$$
and $C_n<0$ is a renormalisation constant that grows logarithmically as $n\to\infty$. As $z_i$ is smooth, by the Cameron-Martin theorem there is a $z_i^\ast\in (C^{-1/2-\kappa}(\mathbb T^3))^\ast$ so that
\begin{align*}
    \mu_n(\B^3_r(z_i))&
    =\frac{1}{Z_{n}} \int_{\B^3_r(0)}\exp\left(-\frac{1}{4}V_n(\phi_n-{z_i,n})-z_i^\ast(\phi)-\frac{1}{2}\|z_i\|_{H_0^1}^2\right)\mu_0(d\phi).
\end{align*}
Applying the binomial theorem to $V_n$ gives that 
\begin{align*}
    V_n(\phi_n-{z_{i,n}})&=\int_{\mathbb T^3} (\phi_n^{:4:}-4\phi_n^{:3:}z_{i,n}+6\phi_n^{:2:} z_{i,n}^2-4\phi_nz_{i,n}^3+z_{i,n}^4) dx\\
    &~~-\int_{\mathbb T^3}( C_n \phi_n^{:2:}-2 C_n\phi_n z_{i,n}+ C_nz_{i,n}^2)dx.
\end{align*}
On $\B^3_r(0)$ we have that there is some constant $K>0$ so that, for any $\kappa'>\kappa$
\begin{align*}
    \left|V_n(\phi_n-z_{i,n})-V_n(\phi_n)+ C_n\|z_{i,n}\|_{L^2}^2\right|&\leq 4\|\phi_n^{:3:}\|_{C^{-3/2-\kappa}}\|z_{i,n}\|_{C^{3/2+\kappa'}}\\
    &+6\|\phi_n^{:2:}\|_{C^{-1-\kappa}}\|z_{i,n}^2\|_{C^{1+\kappa'}}\\
    &+4\|\phi_n\|_{C^{-1/2-\kappa}}\|z_{i,n}^3\|_{C^{1/2+\kappa'}}\\
    &+\|z_{i,n}\|_{L^4}^4\\
    &+ |C_n| \|\phi_n^{:2:}\|_{C^{-1-\kappa}}\|1\|_{C^{1+\kappa'}}\\
    &+2|C_n|\|\phi_n\|_{C^{-1/2-\kappa}}\|z_{i,n}\|_{C^{1/2+\kappa'}}\\
    &\leq K r \log n,
\end{align*}
where the latter follows from the fact that $|C_n|\lesssim \log n$ and several applications of Lemma \ref{lem:dualityBound}  and Lemma \ref{lem:FourierConvergence} to obtain uniform boundedness of $z_{i,n}$ over $n$ in various norms. Thus, from the lower bound
$$ -C_n\|z_{i,n}\|_{L^2}^2-K r \log n\leq V_n(\phi_n-z_{i,n})-V_n(\phi_n)      $$
(recall that $C_n<0$) we can find $r_0=r_0(\|z\|_{L^2})$ sufficiently small and constants $K', K''>0$ so that
\begin{align*}
    \mu_n(\B^3_r(z_i))&\leq \frac{1}{Z_{n}}\int_{\B^3_r(0)}\exp\bigg(-\frac{1}{4}V_n(\phi_n)-K' \log n-z_i^\ast(\phi) -\frac{1}{2}\|z_i\|_{H_0^1}^2\bigg)\mu_0(d\phi)\\
    &\leq \exp(-K' \log n-K'')\mu_n(\B^3_r(0))\leq \exp(-K' \log n-K''),
\end{align*}
where we used that $|z_i^\ast(\phi)|\lesssim r$. 
Therefore, for $z_i\in C^\infty(\T^3),$ we have that $$\mu_n(\mathcal{B}^3_r(z_i))\leq \exp(K r\log n(r)-C_n\|z_{i,n}\|_{L^2}^2-S(z_i))$$ and
    $$\mu_n(\mathcal{B}^3_r(z_i))\geq \exp(-K r\log n(r)-C_n\|z_{i,n}\|_{L^2}^2-S(z_i))$$
    for $i=1,2$. Thus, 
    $$  \frac{\mu_{n(r)}(\mathcal{B}^3_r(z_1))}{\mu_{n(r)}(\mathcal{B}^3_r(z_2))}\leq \exp\bigg(2K r\log n(r) -2C_n\big( \|z_{1,n}\|_{L^2}^2-\|z_{2,n}\|_{L^2}^2  \big)+S(z_2)-S(z_1) \bigg)   $$
    and 
    $$  \frac{\mu_{n(r)}(\mathcal{B}^3_r(z_1))}{\mu_{n(r)}(\mathcal{B}^3_r(z_2))}\geq \exp\bigg(-2K r\log n(r) -2C_n\big( \|z_{1,n}\|_{L^2}^2-\|z_{2,n}\|_{L^2}^2  \big)+S(z_2)-S(z_1) \bigg).   $$
     We conclude by using that $|C_n|\lesssim \log n,$ as well as the assumptions on the $z_i's.$ 
\end{proof}
\begin{remark}
    For an example of $z_1,z_2$ that satisfy the assumptions of Proposition \ref{prop:recover-1}, let $z_1,z_2$ have finite spectral support and equal $L^2$ norms. Another example is given by $z_1, z_2$ with only finitely many non-equal Fourier modes and equal $L^2$ norms. Indeed, this follows from Parseval's theorem.
\end{remark}

\begin{remark}[On the smoothness of $z$] Proposition \ref{prop:recover-1} is proved for smooth $z.$ This requirement was only made to guarantee convergence of the Fourier series $z_n$ in various H\"older or uniform norms. In fact, a more careful power counting, along with Lemma \ref{lem:FourierConvergence}, shows that the latter remains true for $z\in C^{3+\delta}(\T^3)$ with $\delta>0$ arbitrarily small. In particular, it holds for $z_1, z_2\in C^{3+\delta}$ that differ on only finitely many Fourier modes with the same $L^2$ norm.
\end{remark}

For the last part of this section, we consider a generalised small ball ratio over several balls (instead of just $2$). This limit is in fact finite and can be expressed in terms $\Phi_3^4$ action for all pairs of smooth $z_1,z_2.$
\begin{proposition}\label{prop:3d-restore-finiteness}
    Let $z_1,z_2\in C^\infty(\mathbb T^3)$. Let $n(r)$ be any function of $r$ with $\lim_{r\to 0^+}n(r)=\infty$ so that $\lim_{r\to 0^+}\log(n(r))r=0.$ Then the following limit holds
    \begin{align*}
        \lim_{r\to 0^+}&\frac{(\mu_{n(r)}(\mathcal{B}^3_r(z_1)))^3\mu_{n(r)}(\mathcal{B}^3_r(3z_1-2z_2))}{\mu_{n(r)}(\mathcal{B}^3_r(z_2))(\mu_{n(r)}(\mathcal{B}^3_r(2z_1-z_2)))^3}\\
        &=\exp\left(-3S(z_1)-S(3z_1-2z_2)+S(z_2)+3S(2z_1-z_2)\right).
    \end{align*}
\end{proposition}
\begin{proof}
    Following the proof of Proposition \ref{prop:recover-1}, we have for $z\in C^\infty(\mathbb T^3)$ that 
    $$\mu_n(\mathcal{B}^3_r(z))\leq \exp(-C_n\|z_n\|_{L^2}^2)\exp\left(-\frac{1}{4}\|z_n\|_{L^4}^4-\frac{1}{2}\|z\|_{H_0^1}^2\right)\exp(Kr \log n)$$
    and 
    $$\mu_n(\mathcal{B}^3_r(z))\geq \exp(-C_n\|z_n\|_{L^2}^2)\exp\left(-\frac{1}{4}\|z_n\|_{L^4}^4-\frac{1}{2}\|z\|_{H_0^1}^2\right)\exp(-Kr \log n).$$
    Noting that 
    $$\frac{\exp(-3 C_n\|(z_1)_n\|_{L^2}^2)- C_n\|(3z_1-2z_2)_n\|_{L^2}^2)}{\exp(- C_n\|(z_2)_n\|_{L^2}^2-3 C_n\|(2z_1-z_2)_n\|_{L^2}^2)}=1,$$
    we have that 
    \begin{align*}
        \lim_{r\to 0^+}&\frac{(\mu_{n(r)}(\mathcal{B}^3_r(z_1)))^3\mu_{n(r)}(\mathcal{B}^3_r(3z_1-2z_1))}{\mu_{n(r)}(\mathcal{B}^3_r(z_2))(\mu_{n(r)}(\mathcal{B}^3_r(2z_1-z_2)))^3}\\
        &\leq\lim_{r\to 0^+}\exp\left(-3S(z_1)-S(3z_2-2z_2)+S(z_2)+3S(2z_1-z_2))\right)\exp(8Kr\log n)\\
        &=\exp\left(-3S(z_1)-S(3z_2-2z_1)+S(z_2)+3S(2z_1-z_2)\right)
    \end{align*}
    and 
       \begin{align*}
        \lim_{r\to 0^+}&\frac{(\mu_{n(r)}(\mathcal{B}^3_r(z_1)))^3\mu_{n(r)}(\mathcal{B}^3_r(3z_1-2z_1))}{\mu_{n(r)}(\mathcal{B}^3_r(z_2))(\mu_{n(r)}(\mathcal{B}^3_r(2z_1-z_2)))^3}\\
        &\geq\lim_{r\to 0^+}\exp\left(-3S(z_1)-S(3z_1-2z_2)+S(z_2)+3 S(2z_1-z_2)\right)\exp(-8Kr\log n)\\
        &=\exp\left(-3S(z_1)-S(3z_2-2z_1)+S(z_2)+3S(2z_1-z_2)\right).
    \end{align*}
\end{proof}
\begin{remark}
The intuition for Proposition \ref{prop:3d-restore-finiteness} is a generalisation of the fact that the OM ratio is only defined up to an additive constant. If we take higher ordered differences, we can make ``higher ordered" OM ratios that are only defined up to linear, quadratic, and so on, terms. The double difference of a function $f$ between $z_1,z_2$ is $f(z_1)+f(2z_2-z_1)-2f(z_2)$. If $f$ is linear, then this difference is $0$. Therefore if there is a linear divergence, the ratio $$\lim_{r\to 0^+}\frac{\mu(\mathcal{B}^3_r(z_1))\mu(\mathcal{B}^3_r(2z_2-z_1))}{(\mu(\mathcal{B}^3_r(z_2)))^2}$$ is expected to be finite. This second order OM ratio is only defined up to a linear function. The triple difference of a function $f$ between $z_1,z_2$ is $3f(z_1)+f(3z_1-2z_2)-f(z_2)-3f(2z_1-z_2),$ which vanishes if $f$ is quadratic. Therefore, the ``third order" OM functional of Propositon \ref{prop:3d-restore-finiteness} is only defined up to quadratic terms, which is why the quadratic divergence $C_n\int_{\mathbb T^3} z_n^2 dx$ disappears in Proposition \ref{prop:3d-restore-finiteness}. Note that if a measure has a finite ``order $1$" OM ratio, the latter is in addition an ``order $2$" OM ratio and so on. 
\end{remark}

\section{Outlook and conclusion}\label{sec:Conclusion}

As we argued above, the study of OM density functions for $\Phi^4_d$ (as well as other measures arising in CQFT) is interesting and can potentially lead to a deeper understanding of their infinitesimal properties. In this work, we considered (generalised) OM functionals for the $\Phi^4_d$ measures in finite volume for $d=1,2,3.$ 

In $d=1,$ we showed that the classical OM functional coincides, as expected, with the $\Phi^4$ action $S$ \eqref{eq:action}. In $d=2,$ we proposed a natural generalisation and were able to show that, in the setting of Wick renormalisation, the latter agrees with $S$ on $H^1\cap C^{\alpha+\kappa}$ for any $\alpha, \kappa>0.$ We emphasize here that Theorem \ref{thm:main-2d} covers the $P(\Phi)_2$ measures but is also expected to morally hold in fractional dimensions $d,$ within the Da Prato-Debussche regime, as long as the $\Phi^4_d$ is absolutely continuous with respect to the GFF (see e.g. \cite{chandra2025non} for relevant results on ``fractional" $\Phi^4-$measures).

In dimension $3$, we showed that several reasonable generalisations of OM lead to degenerate functionals in the small radius limit. Lastly, we proved that, under certain conditions on the of pairs functions $z_1, z_2,$ small ball ratios centered at $z_i$ are asymptotically equal to the ``correct" $\Phi^4_3$ action in a joint small radius-large frequency limit. Focusing on $d=3,$ we emphasize here that our degeneracy result, Theorem \ref{theorem:main-3d-full}, does not preclude the existence of a non-degenerate OM functional that agrees with $S.$ Indeed, as we highlighted in Section \ref{subsec:OMbackground}, OM functionals are quite sensitive to choices of underlying metrics (in contrast to LDPs whose nature is more topological). 

Theorem \ref{theorem:main-3d-full} relies on elementary arguments, does not make use of non-Gaussian reference measures and only shows that several reasonable choices of metrics lead to degenerate answers. The study of OM functionals for $\Phi^4_3$ under different metrics and with  the aid of more sophisticated tools (such as e.g. the variational framework or drift reference measure of \cite{barashkov2020variational, barashkov2021varphi}) provides an interesting direction for future work. That being said, we emphasize here that OM functionals are less robust than LDPs. Indeed, the $\Phi_3^4$ action has been shown to coincide with the small-temperature LDP rate function \cite{klose2024large} via contraction principles and stochastic quantisation arguments. Nevertheless, the absence of contraction-type principles for OM densities presents an important obstacle in developing more systematic approaches that rely on rough analysis.

\section*{Statements and Declarations}
\begin{itemize}
    \item Conflict of interest: Not applicable. The authors have no competing interests to declare.
    \item Data availability: Not applicable. The authors confirm that no data was generated or analysed during the preparation of this work.
\end{itemize}

 \bibliographystyle{plain}
	   \bibliography{References}

\appendix
\end{document}